\documentclass[11pt,letterpaper]{amsart}

\usepackage{tikz}
\usetikzlibrary{arrows.meta}
\usepackage{url}

\allowdisplaybreaks%

{\theoremstyle{plain}%
 \newtheorem{theorem}{Theorem}
 
 \newtheorem{lemma}{Lemma}%
}
{\theoremstyle{remark}

}
{\theoremstyle{definition}
\newtheorem{definition}{Definition}
\newtheorem{example}{Example}
}

\title[The fibers of the complementary Bell numbers]{The finitude of the fibers of the complementary Bell numbers}
\author{John M.\ Campbell}
\address{Department of Mathematics and Statistics, Dalhousie University,
Halifax, NS B3H 4R2, Canada}
\email{jh241966@dal.ca}
 


\begin{document}

\begin{abstract}
 Subbarao and Verma introduced, in 1999, a number of open problems concerning the sequence $(f(n))_{n \geq 0}$ of 
 complementary Bell numbers, which may be defined via Bell polynomials $B_{n}(x) = \sum_{k=0}^{n} \left\{ \begin{smallmatrix} n \\ 
 k \end{smallmatrix} \right\} x^k$ so that $f(n) = B_{n}(-1)$. Yang [\emph{Electron.\ J.\ Combin.}, 2001] subsequently solved the first two 
 problems from Subbarao and Verma, but the third such problem has remained open, to the best of our knowledge. The first part of 
 this third problem asks whether or not $f(n)$ takes any given value only a finite number of times. We solve this problem in 
 the affirmative, through a combined application of finite difference-based methods, partial Motzkin paths, the completeness of the 
 Tate algebra with respect to the Gauss norm, and Strassmann’s theorem. 
 \end{abstract}

\maketitle

\section{Introduction and background}
 Following the work of Subbarao and Verma \cite{SubbaraoVerma2001}, we write 
\begin{equation}\label{defineestar}
 \prod_{n=2}^{\infty} \big( 1 - n^{-s} \big) = 1 + \sum_{n=2}^{\infty} 
 e^{\ast}(n) n^{-s} 
\end{equation}
 for $\Re(s) > 1$. The integer sequence $\big( e^{\ast}(n) \big)_{n \geq 1}$, for $e^{\ast}(n)$ defined via the Dirichlet series in 
 \eqref{defineestar}, admits a natural number-theoretic and combinatorial interpretation, with 
\begin{equation}\label{interpretestar} 
 e^{\ast}(n) = \sum_{\substack{\text{factorizations $F$ of $n$ into } \\ \text{distinct integers $\geq 2$} }} 
 (-1)^{\text{$\#$ of integers in $F$}}, 
\end{equation} 
 letting it be understood that the factorizations in \eqref{interpretestar} are without regard to any orderings of factors. Being consistent 
 with Subbarao and Verma's notation, we write $f(k)$ in place of $e^{\ast}( p_{i_1} p_{i_2} \cdots p_{i_k} )$ for distinct primes $p_{i_1}$, 
 $p_{i_2}$, $\ldots$, $p_{i_k}$. The integer sequence $(f(k))_{k \geq 0}$, with the understanding that $f(0) = 1$, is referred to as the 
 sequence of \emph{complementary Bell numbers}. Subbarao and Verma proved that 
\begin{equation}\label{limsupf}
 \limsup_{k\to \infty} \frac{\log|f(k)|}{k \log k} = 1, 
\end{equation}
 and concluded with a number of open problems concerning the growth of $f(k)$.   The first two such problems were solved by Yang  
  \cite{Yang2001}, as reviewed below.    The purpose of this paper is to solve the first part of the third problem given by 
 Subbarao and Verma and reviewed below. 
 Based on extant literature related to Subbarao and Verma's work 
 \cite{AmdeberhanDeAngelisMoll2013,DeAngelisMarcello2016,DeWannemackerLaffeyOsburn2007,KataiSubbarao2006,Klazar2003Bell,Klazar2003Counting,Yang2001} \cite[\S3.2]{Mezo2020}, 
 it appears that this third problem has remained open. 

 Let $\left\{ \begin{smallmatrix} n \\ k \end{smallmatrix} \right\} $ denote the Stirling number of the second kind
 with parameters $n$ and $k$, giving the number of set-partitions of an $n$-set into $k$ subsets. 
 Bell numbers and Bell polynomials 
 provide important topics of research in both number theory and combinatorics, 
 and we may express Bell polynomials in terms of Stirling numbers of the second kind so that 
\begin{equation}\label{defineBnx}
 B_{n}(x) = \sum_{k=0}^{n} \left\{ \begin{matrix} n \\ k \end{matrix} \right\} x^k, 
\end{equation}
 with the $x = 1$ case yielding the sequence of Bell numbers. 
 The $x = -1$ case of \eqref{defineBnx} is especially significant in the number-theoretic study of Bell polynomials, 
 as suggested in Mez\H o's text on 
 counting sequences \cite[\S3.2.3]{Mezo2020}. 
 This $x = -1$ case is also central to our work, since 
 $ f(k) = B_k(-1)$. 

 Subbarao and Verma \cite{SubbaraoVerma2001}, in addition to proving the relation in \eqref{limsupf}, concluded with a number of 
 problems concerning or related to $f(k)$. Problem 5.5 from their work
 asks whether or not $f(k)$ changes sign infinitely often. 
 This was later solved in the affirmative by Yang \cite{Yang2001}. 
 Problem 5.6 from Subbarao and Verma 
 asks whether or not $|f(k)|$ is monotonically increasing as $k \to \infty$, for all sufficiently large $k$.
 This was subsequently 
 solved in the negative by Yang. 
 Problem 5.7 from Subbarao and Verma 
 asks, under the assumption that the answer to Problem 5.6 is `no' (as later confirmed by Yang), 
 whether or not $f(k)$ takes any given value only a finite number of times, 
 and also asks whether or not no value is taken more than once apart from $1$ and $-9$. 

 To the best of our knowledge, the first question in Problem 5.7 \cite{SubbaraoVerma2001} has remained open, prior to our current paper. 
 We succeed, 
 through our extensive interactions with GPT-5.6 
 Pro, in solving this first part of Problem 5.7, in the affirmative, i.e., by proving that: For every fixed integer $c$, the relation
\begin{equation}\label{displaymain}
 \#\{ k \geq 0 : f(k) = c \} < \infty 
\end{equation}
 holds. 
 In addition to how this solves an open problem from Subbarao--Verma, 
 the interest in this result 
 and the techniques we introduce to prove this result 
 may be seen in relation to the long-standing conjecture 
 known as \emph{Wilf's conjecture}. 

 The \(c=0\) case of \eqref{displaymain} has been studied extensively in connection with \emph{Wilf's conjecture}, 
 which asserts that 
\(f(k)\neq 0\) for every \(k>2\) (noting that $f(2)=0$). 
 De Wannemacker, Laffey, and Osburn
obtained restrictions on the possible zeros of
\(f(k)\) \cite{DeWannemackerLaffeyOsburn2007}. Amdeberhan, De Angelis,
and Moll \cite{AmdeberhanDeAngelisMoll2013} subsequently gave an alternative proof of the
result that the equation \(f(k)=0\) has at most two solutions. 
 An economical
account of this argument, including corrections to the original proof,
was later given by De Angelis and Marcello
\cite{DeAngelisMarcello2016}. Thus the fiber \(f^{-1}(0)\) was already
known to be finite, whereas our main theorem establishes the finiteness
of $f^{-1}(c)$ for each $c$ in $\mathbb{Z}$.
 
 The study of analytic properties of complementary Bell numbers traces back
 to a 1950 work by Beard \cite{Beard1950}, who introduced and proved a formula for $f(n)$ via Poisson summation. 
 This may be seen as something of an antecedent to Yang's analytic approach toward
 the study of complementary Bell numbers \cite{Yang2001}. In turn, this 
 motivates our proof of \eqref{displaymain} for an arbitrary integer $c$. 

 \section{Preliminaries}\label{secPrelim}
 The preliminaries covered below are required for our purposes. 
 To begin with, we introduce some difference operators that we require. 

 We write $\mathbb{N}$ in place of the set of positive integers, 
 and we write $\mathbb{N}_{0} $ in place of the 
 set of nonnegative integers, and, for a given sequence (defined on $\mathbb{N}_{0}$), 
 we write $\mathcal{S}$ in place
 of the forward-shift operator
 defined so that
\begin{equation}\label{definecalS}
 (\mathcal{S}g)(n) = g(n+1). 
\end{equation} 
 Similarly, we define the difference operator 
 $\Delta_h$ so that 
\begin{equation}\label{Deltahgn} 
 \big( \Delta_h g \big)(n) = g(n + h) - g(n), 
\end{equation} 
 and we recall that the iteration of this operator gives rise to the relation 
 $ \Delta_h^r g(n) = \sum_{j=0}^{r} (-1)^{r-j} \binom{r}{j} g(n + h j)$. 
 We also write $\Delta = \Delta_1$ in place of the usual forward-difference operator. 

\subsection{Partial Motzkin paths}

\begin{definition}\label{def:weighted-Motzkin-array}
 For \(n,k\in\mathbb{N}_0\), let \(\mathcal{M}_{n,k}\) denote 
 the set of \emph{partial Motzkin paths}, i.e., lattice paths from
\((0,0)\) to \((n,k)\), remaining weakly above the horizontal axis and
using only the steps
 $ (1,1)$, $(1,0)$, and $(1,-1)$. 
\end{definition}

\begin{example}
 An illustration of an element in $\mathcal{M}_{5, 2}$ 
 is given in Figure \ref{fig:weighted-partial-Motzkin-path}. 

\begin{figure}[ht]
\centering

\begin{tikzpicture}[
 x=1cm,
 y=1cm,
 pathstep/.style={
 very thick,
 -{Stealth[length=2.2mm,width=1.6mm]}
 },
 vertex/.style={
 circle,
 fill=black,
 inner sep=1.6pt
 },
 steplabel/.style={
 fill=white,
 inner sep=1.2pt,
 font=\small
 }
]

\draw[step=1cm, very thin, gray!45] (0,0) grid (5,5);

\draw[-{Stealth[length=2mm]}, thick]
 (0,0) -- (5.55,0) node[right] {$i$};
\draw[-{Stealth[length=2mm]}, thick]
 (0,0) -- (0,5.55) node[above] {$h$};

\foreach \x in {0,...,5}
 \node[below] at (\x,0) {\small $\x$};

\foreach \y in {1,...,5}
 \node[left] at (0,\y) {\small $\y$};

\coordinate (v0) at (0,0);
\coordinate (v1) at (1,1);
\coordinate (v2) at (2,2);
\coordinate (v3) at (3,2);
\coordinate (v4) at (4,1);
\coordinate (v5) at (5,2);

\draw[pathstep]
 (v0) --
 node[midway, above left, steplabel] {$\beta_0=-1$}
 (v1);

\draw[pathstep]
 (v1) --
 node[midway, above left, steplabel] {$\beta_1=-1$}
 (v2);

\draw[pathstep]
 (v2) --
 node[midway, above, steplabel] {$\gamma_2=1$}
 (v3);

\draw[pathstep]
 (v3) --
 node[midway, above right, steplabel] {$\delta_1=2$}
 (v4);

\draw[pathstep]
 (v4) --
 node[midway, above left, steplabel] {$\beta_1=-1$}
 (v5);

\foreach \v in {v0,v1,v2,v3,v4,v5}
 \node[vertex] at (\v) {};

\node[fill=white, inner sep=2pt] at (3.3,4.25)
 {$P=UUHDU\in\mathcal M_{5,2}$};

\end{tikzpicture}

\caption{A weighted partial Motzkin path of length \(5\) ending at height \(2\).}
\label{fig:weighted-partial-Motzkin-path}
\end{figure}

\end{example}

\begin{definition}\label{defineweight}
 For \(h\in\mathbb{N}_0\), define 
\begin{align*}
 \beta_h & :=
 \begin{cases}
 -(h+1),&\text{if \(h\) is even},\\[1mm]
 -\dfrac{h+1}{2},&\text{if \(h\) is odd},
 \end{cases} \\ 
 \gamma_h & :=h-1, \\ 
 \delta_h & :=
 \begin{cases}
 1,&\text{if \(h\) is even},\\
 2,&\text{if \(h\) is odd}.
 \end{cases}
\end{align*}
 Let \(P\in\mathcal{M}_{n,k}\). The \emph{weight} of a step of \(P\) is assigned so that 
\[
 \begin{array}{rcll}
 (i,h)\longrightarrow(i+1,h+1)
 &:& \beta_h,
 &\text{for an up-step},\\[1mm]
 (i,h)\longrightarrow(i+1,h)
 &:& \gamma_h,
 &\text{for a horizontal step},\\[1mm]
 (i,h+1)\longrightarrow(i+1,h)
 &:& \delta_h,
 &\text{for a down-step}.
 \end{array}
\]
 The \emph{weight} of $P$, denoted by \(\operatorname{wt}(P)\), is the product
of the weights of all its steps, with the convention whereby the empty path is assigned weight \(1\). 
\end{definition}

\begin{example}
 For the path in Figure~\ref{fig:weighted-partial-Motzkin-path}, the successive
step weights are
$\beta_0$, 
$\beta_1$, 
$\gamma_2$, $\delta_1$, and $\beta_1$. 
 Here the down-step from height \(2\) to height \(1\) has weight
\(\delta_1\), since the index is the height of its lower endpoint. Therefore, the equalities among 
$ \operatorname{wt}(P)
 =\beta_0\beta_1\gamma_2\delta_1\beta_1 
 =(-1)(-1)(1)(2)(-1) 
 =-2$ hold. 
 \end{example}

\begin{definition}\label{defineWnk}
 Let 
$
 W_{n,k}:=
 \sum_{P\in\mathcal{M}_{n,k}}\operatorname{wt}(P)$ 
 for nonnegative integers $n$ and $k$, 
 where an empty sum is understood to be equal to $0$, with $
 W_{0,0}=1$ and 
 $ W_{0,k}=0$ for $k \geq 1$. 
\end{definition}

\begin{example}
 The infinite array such that its $(i, j)$-entry is equal to $W_{i-1, j-1}$
 is illustrated below, with 
\begin{equation}\label{infinitearray}
\left(
\begin{array}{cccccccc}
 1 & 0 & 0 & 0 & 0 & 0 & 0 & \cdots \\
 -1 & -1 & 0 & 0 & 0 & 0 & 0 & \cdots \\
 0 & 1 & 1 & 0 & 0 & 0 & 0 & \cdots \\
 1 & 2 & 0 & -3 & 0 & 0 & 0 & \cdots \\
 1 & -1 & -5 & -6 & 6 & 0 & 0 & \cdots \\
 -2 & -11 & -10 & 15 & 30 & -30 & 0 & \cdots \\
 -9 & -18 & 16 & 120 & 30 & -270 & 90 & \cdots \\
 \vdots & \vdots & \vdots & \vdots & \vdots & \vdots & \vdots & \ddots \\
\end{array}
\right).
\end{equation}
\end{example}

 Recall that $W_{0, 0} = 1$ and that $W_{0, k} = 0$ for integers $k \geq 1$. As suggested in the numerical evaluations displayed in 
 \eqref{infinitearray}, we have that $W_{n, k} = 0$ for $k > n$. 
 These base cases lead us to the known recurrences 
\begin{equation}\label{Wrec1} 
 W_{n+1, 0} = \gamma_0 W_{n, 0} +\delta_0 W_{n, 1} 
\end{equation}
 and 
\begin{equation}\label{Wrec2} 
 W_{n+1, k} = \beta_{k-1} W_{n, k-1} + \gamma_k W_{n, k} + \delta_k W_{n, k+1}, 
\end{equation}
 with both \eqref{Wrec1} and \eqref{Wrec2} being given in An's thesis
 \cite[p.\ 19]{An2010}. 

 An \cite[\S1.3--2.2]{An2010} showed how complementary Bell numbers can be expressed via the enumeration of Motzkin paths, by 
 showing, in an equivalent way, that 
\begin{equation}\label{connectfW} 
 f(n) = W_{n, 0}
\end{equation}
 for all nonnegative integers $n$. The relation in 
 \eqref{connectfW} is implicit in An's thesis \cite[p.\ 19]{An2010} and given explicitly
 in a closely related preprint by An\footnote{See \url{https://arxiv.org/abs/0807.2150v1}.}
 and can be can also be shown using the given recurrences in 
 \eqref{Wrec1} and \eqref{Wrec2}, via the use of exponential generating functions. 
 The relation in \eqref{connectfW} is key for our purposes, in terms of our use of 
 the result from An highlighted as Lemma \ref{Anlemma} below. 
 
 Recalling the definition of $\Delta_h$ in \eqref{Deltahgn}, we let it be understood that $ \Delta_h$ acts on the bivariate function $W_{n, 
 k}$ with respect to the initial index $n$, i.e., so that 
 $ \Delta_h W_{n, k} = W_{n + h, k} - W_{n, k}$. 

\begin{lemma}\label{Anlemma}
 (An, 2010) For $n, k \in \mathbb{N}_{0}$ and for $r, t \in \mathbb{N}$, 
 the congruence
\begin{equation}\label{Andisplay}
 \Delta_{6(2t-1)}^r W_{n, k} \equiv 0 \pmod{2^r} 
\end{equation}
 holds \cite[p.\ 20]{An2010}. 
\end{lemma}

\subsection{2-adic analysis}\label{prelim2adic}
 Letting $q$ denote a nonzero rational number, we write $q = 2^{m} \frac{a}{b}$ for $m \in \mathbb{Z}$ and for odd integers $a$ and 
 $b$. The value $m$ is determined uniquely and may be denoted with $\nu_2(q)$ and is referred to as the \emph{normalized additive 
 $2$-adic valuation of $q$}, and we adopt the convention whereby $\nu_2(0) = +\infty$. 
 The \emph{$2$-adic absolute value} on $\mathbb{Q}$ may then be defined so that 
\begin{equation}\label{2adicabs}
 \left| q \right|_{2} := \begin{cases} 
 0, & \text{if $q = 0$;} \\ 
 2^{-\nu_2(q)}, & \text{if $q \neq 0$.} 
 \end{cases} 
\end{equation}
 The definition in \eqref{2adicabs} then gives rise to the \emph{2-adic metric}
 defined so that $d_{2}(x, y) = |x - y|_{2}$. 
 The field $\mathbb{Q}_{2}$ of \emph{2-adic numbers} may then be defined as the completion of 
 $\mathbb{Q}$ as a metric space with respect to $d_2$, 
 and we let the unity element in $\mathbb{Q}_{2}$ be denoted as $1$. 
 We may then define the ring $\mathbb{Z}_{2}$ of \emph{2-adic integers}
 so that $\mathbb{Z}_{2} := \{ x \in \mathbb{Q}_{2} : |x|_{2} \leq 1 \}$, noting that this is a ring with unity. 

 For each nonzero element $x \in \mathbb{Q}_{2}$, there exists a unique value $m \in \mathbb{Z}$ and a unique $2$-adic unit $u \in 
 \mathbb{Z}_{2}^{\times}$ satisfying $x = 2^{m} u$. We then let $\nu_2(x) := m$, 
 again with the convention that $\nu_2(0) = +\infty$. 

 The \emph{one-variable Tate algebra over $\mathbb{Q}_2$} may then be defined so that 
\begin{equation}\label{defineTate}
 \mathbb{Q}_{2}\langle z \rangle 
 = \left\{ \sum_{j=0}^{\infty} b_{j} z^{j} : b_{j} \in \mathbb{Q}_{2}, \, \left| b_{j} \right|_{2} \to 0 \right\}. 
\end{equation}
 For an element
 $F(z) = \sum_{j=0}^{\infty} b_{j} z^{j}$ in the family on the right-hand side of 
 \eqref{defineTate}, with $b_{j} \in \mathbb{Q}_{2}$ for a nonnegative integer $j$, 
 the \emph{Gauss norm} of $F$ is defined so that 
\begin{equation}\label{defineGauss}
 \| F \|_{\operatorname{G}} := \max_{j \geq 0} \left| b_{j} \right|_{2}. 
\end{equation}
 Observe that the maximum value purportedly defined on the right of \eqref{defineGauss}
 exists as a consequence of the sequence $\big( |b_j|_{2} : j \in \mathbb{N}_{0} \big)$
 tending to $0$. The Gauss norm may equivalently be defined, for nonzero $F$, so that 
\begin{equation}\label{alternativenorm} 
 \| F \|_{\operatorname{G}} = 2^{-\min\{ \nu_2(b_j) : b_j \neq 0 \}}. 
\end{equation}
 The norm $\| \cdot \|_{\operatorname{G}}$
 satisfies the ultrametric triangle inequality, i.e., so that:
 For $F = F(z)$ and $G = G(z)$ in $\mathbb{Q}_{2}\langle z \rangle$, the relation 
\begin{equation}\label{ultrametric}
 \| F + G \|_{\operatorname{G}} \leq \max\{ \| F \|_{\operatorname{G}}, \| G \|_{\operatorname{G}} \} 
\end{equation}
 holds. 

 Let $ G(z) = \sum_{j=0}^{\infty} b_{j} z^{j} \in \mathbb{Q}_{2}\langle z \rangle $
 and let $y \in \mathbb{Z}_{2}$. Now, consider partial sums of the form 
 $ G_{N}(y) := \sum_{j=0}^{N} b_{j} y^{j}$. 
 Letting $M$ be an integer exceeding $N$, 
 the ultrametric inequality in \eqref{ultrametric}
 gives us, by exploiting the relation $|y|_{2} \leq 1$, that 
\begin{align*}
 \left| G_{M}(y) - G_{N}(y) \right|_{2}
 & = \left| \sum_{j = N + 1}^{M} b_{j} y^{j} \right|_{2} \\ 
 & \leq \max_{N < j \leq M} \left| b_{j} y^{j} \right|_{2} \\ 
 & \leq \max_{N < j \leq M} \left| b_{j} \right|_{2}. 
\end{align*}
 Since we let $G$ be a member of $\mathbb{Q}_{2}\langle z \rangle$, 
 the definition of the Tate algebra gives us that $|b_{j}|_{2}$ approaches $0$ as $j \to \infty$. 
 So, we find that $\big( G_{N}(y) \big)_{N \geq 0}$ 
 is a Cauchy sequence	 with respect to $| \cdot |_{2}$ in $\mathbb{Q}_{2}$. 
 Since $\mathbb{Q}_{2}$ is complete (by definition), 
 we obtain that $\big( G_{N}(y) \big)_{N \geq 0}$ has a limit 
 in $\mathbb{Q}_{2}$. This allows us to define
\begin{equation}\label{defineGy} 
 G(y) := \lim_{N \to \infty} \sum_{j=0}^{N} b_{j} y^{j} 
\end{equation}
 as an element in $\mathbb{Q}_{2}$, recalling the assumption that $y \in \mathbb{Z}_{2}$. 

 Again for $y \in \mathbb{Z}_{2}$, letting $G(y)$ be as above, an application of the ultrametric inequality gives us that $ 
 \left| G_{N}(y) \right| $ $ \leq $
 $ \max_{0 \leq j \leq N} \left| b_{j} y^{j} \right|_{2} $
 $ \leq $
 $ \max_{0 \leq j \leq N} \left| b_{j} \right|_{2} $
 $ \leq $
 $ \| G \|_{\operatorname{G}}$. 
 By then setting $N \to \infty$, we find that 
\begin{equation}\label{dominateGauss}
 \left| G(y) \right|_{2} \leq \| G \|_{\operatorname{G}}. 
\end{equation} 

 A key to our construction in Section \ref{secmain}
 relies on the completeness of the Tate algebra with respect to $\| \cdot \|_{\operatorname{G}}$. 
 For background on and a derivation of this result (given in an equivalent way), 
 we refer to Bosch's text on formal and 
 rigid geometry \cite[p.\ 14]{Bosch2014}.

\subsection{Strassmann's theorem}
 For background related to our formulation of Strassmann's theorem required as a central component of our construction, we refer to 
 the appropriate text by 
 Cohen \cite[\S4.5.1]{Cohen2007}, which provides
 a slightly stronger version of the following formulation
 that we require. 

 \ 

\noindent {\bf Strassmann's Theorem:} Let $ G(z) = \sum_{j=0}^{\infty} b_{j} z^{j} \in \mathbb{Q}_{2}\langle z \rangle $ be nonzero. Set 
 $ M(G) := \max_{j \geq 0} |b_{j}|_{2} $ and $ N(G) := \max\{ j \geq 0 : |b_j|_{2} = M(G) \}$. Then $G$ has at most $N(G)$ distinct zeros 
 in $\mathbb{Z}_{2}$. 

\section{Main result}\label{secmain}
 Setting $t = 1$ and $k = 0$ in \eqref{Andisplay}, the above formulation of An's lemma, as in Lemma \ref{Anlemma}, gives us that 
\begin{equation}\label{nu2Delta6}
 \nu_2\big( \Delta_{6}^{r} f(n) \big) \geq r. 
\end{equation}
 This leads us toward the following companion to \eqref{nu2Delta6} required for our purposes. Informally, we need a greater lower bound, 
 relative to \eqref{nu2Delta6}, due to the growth/behavior of $\nu_{2}(r!)$, since we need to consider the 2-adic valuations of the 
 powers of $z$ arising in the expansion of expressions of the form $\binom{z}{r} \Delta_{12}^r f(a)$. 

\begin{lemma}\label{nu2Delta12}
 For all integers $n \geq 0$ and $r \geq 1$, the relation $ \nu_2\big( \Delta_{12}^{r} f(n) \big) \geq 2 r $ holds.
\end{lemma}

\begin{proof}
 From the definition in \eqref{Deltahgn}, we find that $\Delta_{12} = \Delta_{6}^{2} + 2 \Delta_{6}$. The iteration of $\Delta_{12}$ then 
 gives rise to the expansion 
\begin{equation}\label{Delta12r}
 \Delta_{12}^{r} = \sum_{j=0}^{r} \binom{r}{j} 2^{r-j} \Delta_{6}^{r+j}. 
\end{equation}
 We then apply each side of \eqref{Delta12r} to the sequence of complementary Bell numbers, writing 
\begin{equation}\label{Delta12tof} 
 \Delta_{12}^{r} f(n) = \sum_{j=0}^{r} \binom{r}{j} 2^{r-j} \Delta_6^{r+j} f(n). 
\end{equation}
 From An's lemma and its reformulation in \eqref{nu2Delta6}, we find that 
\begin{equation}\label{divideterms}
 2^{r + j} \mid \Delta_{6}^{r+j} f(n). 
\end{equation}
 From \eqref{Delta12tof} and \eqref{divideterms} together, we find that 
 each term on the right-hand side of \eqref{Delta12tof} is divisible by $2^{2r}$, and hence 
 the desired result. 
\end{proof}

\begin{lemma}\label{lemmaFf}
 For each $a \in \{ 0, 1, \ldots, 11 \}$   there is a power series $F_{a}(z)$ in $\mathbb{Q}_{2}\langle z \rangle$   such that:   For all $m \in  
  \mathbb{N}_{0}$, the equality   $F_{a}(z) |_{z = m} = f(a + 12 m)$ holds. 
\end{lemma}

\begin{proof}
 Fix an integer parameter $a \in \{ 0, 1, \ldots, 11 \}$. Let the operator $\Delta'$ act on the sequence of expressions of the form $f(a + 
 12 m)$ for $m \in \mathbb{N}_{0}$ with respect to $m \in \mathbb{N}_{0}$ so that
\begin{equation}\label{Deltaprime}
 \Delta' f(a + 12 m) = f(a + 12 (m+1)) - f(a + 12m), 
\end{equation}
 we then rewrite $(\Delta')^{r} f(a + 12 m) \big|_{m=0}$ in terms of $\Delta_{12}$, i.e., with
\begin{equation}\label{rewriteDelta}
 (\Delta')^{r} f(a + 12 m) \big|_{m=0} = \Delta_{12}^{r} f(a + m) \big|_{m=0}. 
\end{equation}
 By then letting it be understood that $\Delta_{12}$ acts on the $f$-sequence according to the original definition in \eqref{Deltahgn}, we 
 may rewrite \eqref{rewriteDelta} so that 
\begin{equation}\label{actseq}
 (\Delta')^{r} f(a + 12 m) \big|_{m=0} = \Delta_{12}^{r} f(a). 
\end{equation}
 
 Let $z$ be an indeterminate. For a given integer $r \geq 1$, we write $\binom{z}{r} = \frac{z(z-1) \cdots (z-r + 1)}{r!}$, with 
 $\binom{z}{0}$ being equal to the unity element in the Tate algebra in \eqref{defineTate}. Since $z (z-1) \cdots (z - r + 1)$ is in 
 $\mathbb{Z}[z]$, by expanding $ \binom{z}{r} \, \Delta_{12}^{r} f(a)$ in powers of $z$, 
 we find that the 2-adic valuation of each of the resulting coefficients
 is at least $\nu_2\big( \Delta_{12}^{r} f(a) \big) - \nu_2(r!)$. 
 Legendre's formula gives us that 
\begin{equation}\label{fromLegendre} 
 \nu_2(r!) = \sum_{\ell \geq 1} \left\lfloor \frac{r}{2^{\ell}} \right\rfloor, 
\end{equation}
 and, with the assumption that $r \geq 1$, the relation in  \eqref{fromLegendre} gives us that  $\nu_{2}(r!) \leq r - 1$. This together with 
 Lemma \ref{nu2Delta12} give us that every coefficient in  the expansion of  $ \binom{z}{r} \, \Delta_{12}^{r} f(a)$ in powers of $z$  has a 
 $2$-adic valuation at least 
\begin{equation}\label{boundnuDelta}
 \nu_{2}\left( \Delta_{12}^{r} f(a) \right) - \nu_{2}(r!)
 \geq 2 r - (r-1) = r + 1.
\end{equation}
 From \eqref{alternativenorm} and \eqref{boundnuDelta}, we obtain the bound 
\begin{equation}\label{boundGauss}
 \left\| \binom{z}{r} \, \Delta_{12}^{r} f(a) \right\|_{\operatorname{G}} \leq 2^{-(r+1)} 
\end{equation}
 for $r \geq 1$.
 From the ultrametric triangle inequality in \eqref{ultrametric}
 together with the bound in \eqref{boundGauss}, we find that 
\begin{align}
\begin{split}
 \left\| \sum_{r=N+1}^{M} \binom{z}{r} \, \Delta_{12}^{r} f(a) \right\|_{\operatorname{G}}
 & \leq \max_{N+1 \leq r \leq M} \left\| \binom{z}{r} \, \Delta_{12}^{r} f(a) \right\|_{\operatorname{G}} \\ 
 & \leq 2^{-(N+2)}. 
\end{split}\label{givesCauchy}
\end{align}

 Define 
\begin{equation}\label{defineSaM}
 S_{a, M}(z) = \sum_{r=0}^{M} \binom{z}{r} \, \Delta_{12}^{r} f(a). 
\end{equation}
 The inequalities in \eqref{givesCauchy} then give us that 
 the sequence $\big( S_{a, M} \big)_{M \geq 0}$ is a Cauchy sequence with respect to $\| \cdot \|_{\operatorname{G}}$. 
 Since $\mathbb{Q}_{2}\langle z \rangle$ is complete with respect to $\| \cdot \|_{\operatorname{G}}$, 
 there exists a unique element $F_{a}(z)$ in 
 $\mathbb{Q}_{2}\langle z \rangle$
 such that the partial sums in \eqref{defineSaM}
 converge to $F_{a}(z)$
 with respect to the Gauss norm. 
 This allows us to write 
\begin{equation}\label{defineFaz}
 F_{a}(z) := \sum_{r=0}^{\infty} \binom{z}{r} \, \Delta_{12}^{r} f(a). 
\end{equation} 
 From the completeness of $ \mathbb{Q}_{2}\langle z \rangle$ with 
 respect to $\| \cdot \|_{\operatorname{G}}$, 
 we can conclude that the power series expansion used to define $F_{a}(z)$
 converges to an element in $\mathbb{Q}_{2}\langle z \rangle$. 

 Recall (see Section \ref{prelim2adic}) that: For $y \in \mathbb{Z}_{2}$ and for $G = G(z) \in \mathbb{Q}_{2}\langle z \rangle$, the limit on 
 the right of \eqref{defineGy} exists 
 in such a way so that we may define $G(y) \in \mathbb{Q}_{2}$
 in the manner indicated above. 
 This allows us to define 
\begin{equation}\label{Facolon}
 \mathcal{F}_{a}\colon \mathbb{Z}_{2} \to \mathbb{Q}_{2} 
\end{equation}
 so that 
 $ \mathcal{F}_{a}(y) := F_a(y) $ 
 for $y \in \mathbb{Z}_{2}$. 
 Using \eqref{dominateGauss}, we find that 
\begin{align}
\begin{split}
 \left| S_{a, N}(y) - \mathcal{F}_{a}(y) \right|_{2} 
 & = \left| \left( S_{a, N}(z) - F_{a}(z) \right) |_{z = y} \right| \\ 
 & \leq \| S_{a, N} - F_{a} \|_{\operatorname{G}}. 
\end{split}\label{split2G}
\end{align}
 As demonstrated above, the partial sums in \eqref{defineSaM}
 converge to $F_{a}(z)$
 with respect to $\| \cdot \|_{\operatorname{G}}$, 
 so that \eqref{split2G}
 allows us to conclude that $S_{a, N}$ approaches $\mathcal{F}_{a}(y)$ with respect to $\| \cdot \|_{2}$
 as $N \to \infty$. This allows us to write 
\begin{equation}\label{definecalF} 
 \mathcal{F}_{a}(y) = \sum_{r=0}^{\infty} \binom{y}{r} \Delta_{12}^{r} f(a), 
\end{equation}
 again for $y \in \mathbb{Z}_{2}$. 

 Now, we choose an element $m \in \mathbb{N}_{0} \subseteq \mathbb{Z}_{2}$ in the domain in \eqref{Facolon}. Observe that 
 $\binom{m}{r}$ vanishes for each integer $r > m$, 
 and this along with the definition in \eqref{definecalF} give us the finite sum expansion such that 
\begin{equation}\label{finiteforF}
 \mathcal{F}_{a}(m) = \sum_{r=0}^{m} \binom{m}{r} \, \Delta_{12}^{r} f(a). 
\end{equation}
 Recalling the operator $\mathcal{S}$ defined in Section \ref{secPrelim}, 
 and recalling that we are letting 
 $a \in \{ 0, 1, \ldots, 11 \}$ be understood as a fixed parameter, 
 we let it be understood that $\mathcal{S}$
 acts on the sequence of expressions of the form $f(a + 12 j)$
 for $j \in \mathbb{N}_{0}$ so that 
 $\mathcal{S} f(a + 12 j) = f(a + 12 (j+1))$. 
 From \eqref{Deltaprime}, we then find that 
 	 $ \Delta' f(a + 12j) = (\mathcal{S} - \operatorname{id}) \, f(a + 12 j)$, 
 for the identity operator $\operatorname{id}$, and with the understanding that sums/differences of operators
 defined on the same sequence are to be understood in a pointwise fashion. 
 So, with the underanding that we are restricting our attention to operators
 on the sequence $\big( f(a + 12 j) \big)_{j \geq 0}$, 
 we find that $\mathcal{S} = \operatorname{id} + \Delta'$. 
 Since $\operatorname{id}$ and $\Delta'$ commute, the binomial theorem gives us that 
\begin{equation}\label{fromcommutative} 
 \mathcal{S}^{m} = \sum_{r=0}^{m} \binom{m}{r} \left( \Delta' \right)^{r}. 
\end{equation}
 We proceed to apply \eqref{fromcommutative} to both 
 sides of $\big( f(a + 12 j) \big)_{j \geq 0}$, and by then setting $j=0$, we find 
 (recalling \eqref{definecalS}) that 
\begin{align}
\begin{split}
 f(a + 12 m) & = \mathcal{S}^{m} f(a + 12 j) \big|_{j=0} \\ 
 & = \sum_{r=0}^{m} \binom{m}{r} \left( \Delta' \right)^{r} 
 f(a + 12 j) \big|_{j=0}. 
\end{split}\label{binomialj0}
\end{align}
 From \eqref{actseq} and \eqref{binomialj0} together, we find that 
\begin{equation}\label{binomialf}
 f(a + 12m) = \sum_{r=0}^{m} \binom{m}{r} \, \Delta_{12}^{r} f(a). 
\end{equation}
 A comparison between \eqref{finiteforF} and \eqref{binomialf} 
 allows us to conclude that 
 $\mathcal{F}_{a}(m) = f(a + 12m)$, recalling that both $m \in \mathbb{N}_{0}$ and $a \in \{ 0, 1, \ldots, 11 \}$
 are arbitrary. 
\end{proof}

\begin{lemma}\label{lemmainterpolate}
 For every $a \in \{ 0, 1, \ldots, 11 \}$, the interpolating function $F_{a}$ 
 is nonconstant. 
\end{lemma}

\begin{proof}
 Since $f(n) = B_{n}(-1)$, the usual recurrence for Bell polynomials
 gives us that 
 $ f(n) = - \sum_{j=0}^{n-1} \binom{n-1}{j} f(j)$ 
 for positive integers $n$. 
 By appealing to the numerical values among 
 $$ (f(0), \ldots, f(11)) = (1, -1, 0, 1, 1, -2, -9, -9, 50, 267, 413, -2180) $$
 and 
\begin{multline*}
 (f(12), \ldots, f(17))
 = (-17{\,}731, -50{\,}533, 110{\,}176, 1{\,}966{\,}797, \\ 9{\,}938{\,}669, 8{\,}638{\,}718) 
\end{multline*}
 and 
\begin{multline*}
 (f(18), \ldots, f(23))
 = (-278{\,}475{\,}061, -2{\,}540{\,}956{\,}509, -9{\,}816{\,}860{\,}358, \\ 
 27{\,}172{\,}288{\,}399, 725{\,}503{\,}033{\,}401, 5{\,}592{\,}543{\,}175{\,}252), 
\end{multline*}
 we find that 
 $ \max_{0\leq a < 12} |f(a)| = 2180 $
 and that 
 $ \min_{0\leq a < 12} |f(a + 12)| = 17{\,}731$. 
 Hence, for every $a \in \{ 0, \ldots, 11 \}$, we have that 
 $ |F_{a}(1)| = |f(a + 12)| > |f(a)| = |F_{a}(0)|$. 
 Thus $F_{a}(0) \neq F_{a}(1)$, so $F_{a}$ is nonconstant. 
\end{proof}

 This leads us toward our main result, as below. 

\begin{theorem}\label{maintheorem}
 For each $c \in \mathbb{Z}$, the fiber
$ \left\{ n \in \mathbb{N}_{0} : f(n) = c \right\} $ is finite. 
\end{theorem}

\begin{proof}
 Let $c$ be a fixed integer. For each value $ a$ in $ \{ 0, 1, \ldots, 11 \}$, 
 we define the $2$-adic analytic function 
\begin{equation}\label{defineGac} 
 G_{a, c}(z) := F_{a}(z) - c. 
\end{equation}
 By Lemma \ref{lemmainterpolate}, 
 the interpolating function $F_{a}$ is nonconstant. 
 This allows us to conclude that 
 the element of the Tate algebra given on either side of the equality in \eqref{defineGac} 
 is not the zero power series in $\mathbb{Q}_{2}\langle z \rangle$. 
 Strassmann's theorem then allows us to conclude that the zero set
\begin{equation}\label{defineZac}
 Z_{a, c} := \{ z \in \mathbb{Z}_{2} : G_{a, c}(z) = 0 \} 
\end{equation}
 is finite. 
 Now, let $m \in \mathbb{N}_{0}$. 
 Lemma \ref{lemmaFf} then gives us the biconditional equivalences such that 
\begin{equation}\label{biconditional}
 f(a+12m) = c 
 \ \Longleftrightarrow 
 \ F_{a}(m) - c = 0 
 \ \Longleftrightarrow \ m \in Z_{a, c} \subseteq \mathbb{Z}_{2}
\end{equation}
 So, from the inclusion of $\mathbb{N}_{0}$ in $\mathbb{Z}_{2}$, 
 along with the finitude the family defined in \eqref{defineZac}, together with the equivalences in 
 \eqref{biconditional}, we find that the set 
\begin{equation}\label{firstfinite}
 \{ m \in \mathbb{N}_{0} : f(a + 12 m) = c \} 
\end{equation}
 is finite: 
 By then using the decomposition of $\mathbb{N}_{0}$ into residue classes modulo $12$, 
 we find that 
\begin{equation}\label{secondfinite} 
 \{ n \in \mathbb{N}_{0} : f(n) = c \} 
 = \bigcup_{a=0}^{11} \{ a + 12 m : m \in \mathbb{N}_{0}, f(a + 12 m) = c \}.
\end{equation}
 From the finitude of \eqref{firstfinite}, 
 the right-hand side of \eqref{secondfinite} 
 is a finite union of finite sets, and is, consequently, finite. 
\end{proof}

 A similar approach can also be used to obtain a stronger result giving us the \emph{uniform} boundedness of the fibers of the sequence 
 of complementary Bell numbers. For brevity, we omit a full derivation of this. 
 Observe that Theorem \ref{maintheorem} also gives us (in an equivalent way) that 
 the sequence of absolute values of complementary Bell numbers tends to infinity. 

\subsection*{Acknowledgements}
 The author acknowledges extensive interactions with GPT-5.6 
 Pro during the exploratory and proof-development stages of this work. All AI-generated suggestions were 
 substantially revised, corrected, and independently verified by the author, who assumes full responsibility for the mathematical content.

\bibliographystyle{plain}
\bibliography{seprefe}

\end{document}